\documentclass[12pt,a4paper]{article}
\usepackage{amsmath,amssymb,amsthm, bm}
\usepackage{enumerate}
\usepackage[left=2cm,top=3cm,right=2cm,bottom=3cm]{geometry}
\usepackage{color}
\usepackage{relsize}%

\newtheorem{theorem}{Theorem}[section]
\newtheorem{corollary}[theorem]{Corollary}
\newtheorem{lemma}[theorem]{Lemma}
\newtheorem{remark}[theorem]{Remark}
\newtheorem{proposition}[theorem]{Proposition}
\newtheorem{example}[theorem]{Example}
\newtheorem{definition}[theorem]{Definition}
\numberwithin{equation}{section}

\newcommand{\C}{{\ensuremath{\mathbb{C}}}}
\newcommand{\Cm}{{\ensuremath{\C^{m\times n}}}}
\newcommand{\Cn}{{\ensuremath{\C^{n\times m}}}}
\newcommand{\Cnn}{{\ensuremath{\C^{n\times n}}}}
\newcommand{\Cmm}{{\ensuremath{\C^{m\times m}}}}
\newcommand{\Crr}{{\ensuremath{\C^{r\times r}}}}

\newcommand{\Ra}{{\ensuremath{\cal R}}}
\newcommand{\Nu}{{\ensuremath{\cal N}}}
\newcommand{\rk}{{\ensuremath{\rm rank}}}
\newcommand{\core}{\mathrel{\text{\textcircled{$\#$}}}} 

\newcommand{\rank}{{\ensuremath{\rm rank}}}

\newcommand{\ind}{{\ensuremath{\text{\rm Ind}}}}

\newcommand{\st}{\stackrel{*}\le} 
\newcommand{\co}{\stackrel{\core}\le}   

\newcommand{\cm}{{\text{\tiny \rm GM}}}

\newcommand{\cmbis}{\C_n^{\cm}}

\begin{document}

\author{D.E. Ferreyra\thanks{Universidad Nacional de R\'io Cuarto, CONICET, FCEFQyN, RN 36 KM 601, 5800 R\'io Cuarto, C\'ordoba, Argentina. E-mail: \texttt{deferreyra@exa.unrc.edu.ar, flevis@exa.unrc.edu.ar, pmoas@exa.unrc.edu.ar}} ,  F.E. Levis$^*$,  Saroj B. Malik\thanks{School of Liberal Studies, Ambedkar University, Kashmere Gate, Delhi, India. E-mail:  \texttt{saroj.malik@gmail.com}} , R.P. Moas$^*$}

\title{One sided Star and Core orthogonality of matrices}
\date{}
\maketitle

\begin{abstract}
We investigate two one-sided orthogonalities of matrices, the first of which is left (right) $*$-orthogonality for rectangular matrices and the other is left (right) $\tiny\core$-orthogonality of index $1$ matrices. We obtain some basic results for these matrices, their canonical forms, and characterizations. Also, relations between left (right) orthogonal matrices and parallel sums are investigated. Finally under these one-sided orthogonalities we explore the conditions of additivity of the Moore-Penrose inverse and the core inverse. 
\end{abstract}



\section{Introduction}
\subsection{Mise-en-scéne}

The concept of $*$-orthogonality was  introduced by Hestenes  \cite{Hes} in order to develop a spectral theory for rectangular matrices.  Since then, it has been extensively used in studying various problems like the additive property of the Moore-Penrose inverse, rank additivity, partial orders and many more problems, which further  have been used successfully for parallel sums and shorted matrices (via parallel sums) \cite{HaSt, HaOmSmSt}. The $*$-orthogonality also plays an important role  in the matrix version of Cochran's theorem, and it's newer versions \cite{MiBhMa}. Recently, in \cite{FeMa3} the authors introduced  the concept of relative EP matrix of a rectangular matrix relative to a partial isometry matrix (or, in short, $T$-EP) by using the Moore-Penrose inverse. It was proved that the $*$-orthogonality is a sufficient condition for the sum of two rectangular matrices $T$-EP to be  $T$-EP. 
In \cite{FeMa2}  the same authors, studied the concepts of Core orthogonality and Strongly core orthogonality of two square matrices of the index  1, by using the core inverse \cite{BaTr} instead of the Moore-Penrose inverse.  Two recent works in this direction can be found in \cite{MoDoKuMa, LiWaWa}.  

From school geometry it is known that if a line ${\ell}_1$ is perpendicular to another line ${\ell}_2$ and ${\ell}_2$ is perpendicular to ${\ell}_3$, then either the lines ${\ell}_1$ and ${\ell}_3$ are coincident or are parallel. It's well known that the notion of orthogonal matrices is a generalization of this (perpendicularity) to matrices. How do we view parallelism in matrices. One way and perhaps the best way is to think of the parallel sums. It may be worth to note that the parallel sums originally arose in the study of network synthesis \cite{AnDu}. Recall the parallel sums have various applications, for instance,  in networking, electrical engineering as well as in statistical problems connected with the estimation of parameters in a linear system \cite{MiOd, MiBhMa}.

The present work deals to the notion  of left (right) $*$-orthogonality between two rectangular matrices, that is,  the one-sided version of $*$-orthogonality. We then develop their connection with parallel sums and the  additivity  property of the Moore-Penrose inverse.  Correspondingly, for index $1$ matrices, one-sided core orthogonality has also been explored for its properties and properties similar to one-sided $*$-orthogonality.  

The paper is organized as follows. In Section 2, we introduce the notion  of left (right) $*$-orthogonality between two rectangular matrices. We present their canonical forms by using the Singular Value Decomposition. In particular, we derive interesting connections with parallel sums and the Moore-Penrose inverse of the sum of two matrices. In Section 3, we study the one-sided version of the core orthogonality recently introduced in \cite{FeMa2}. In particular, we obtain some new results on the core-additivity. 

\subsection{Notation and preliminaries results}

For $A\in\Cm$, the symbols $A^*$,  $\Ra(A)$, $\Nu(A)$, and $\rk(A)$,  will stand for the conjugate transpose, column space, null space, and rank of $A$, respectively.  Further, $I_n$ will refer the  identity matrix of order $n$. The index of $A\in \Cnn$, denoted by $\ind(A)$, is the smallest nonnegative integer $k$ such that $\rk(A^k)=\rk(A^{k+1})$. 
For a subspace $\mathcal{M}$ we denote by $\mathcal{M}^\perp$   the orthogonal complement of $\mathcal{M}$.
 
Usually,  $A^\dag \in \Cn$, will stand for the Moore-Penrose inverse of $A$, i.e., the unique matrix satisfying the four equations
\[AA^\dag A=A, \quad A^\dag AA^\dag=A^\dag,
\quad (AA^\dag)^*=AA^\dag, \quad (A^\dag A)^*=A^\dag A.\]
It is well-known that the Moore-Penrose inverse  can be used to represent orthogonal projectors. 
Denote  $P_A=AA^{\dag}$, $Q_A=A^{\dag}A$, $\overline{P}_A=I_m-P_A$, and $\overline{Q}_A=I_n-Q_A$, which are the orthogonal projectors onto  $\Ra(A)$, $\Ra(A^*)$, $\Nu(A^*)
$, and $\Nu(A)$, respectively. Also, a matrix $A^- \in \Cn$ that satisfies the equality $AA^-A = A$ is called an $g$-inverse (or inner inverse) of $A$, and the set of all $g$-inverses of $A$ is denoted by $A\{1\}$. 

In this paper we will use the core inverse $A^{\core}$ of a matrix $A\in \Cnn$ introduced by Baksalary and Trenkler in \cite{BaTr}. Recall that $A^{\core}$ denotes the unique matrix (whenever it exists) defined by the two conditions 
\[AA^{\core}=P_A, \quad \Ra(A^{\core})\subseteq \Ra(A).\]

It is well-known that the core inverse of $A$ exists if and only if $\ind(A)\leq 1$, in which case $A$ is also called a group matrix. The symbol $\cmbis$ will stand for the subset of $\Cnn$ consisting of group matrices. 

We finish this subsection with some auxiliary lemmas,   which will be used in this paper.

\begin{lemma}\label{MP blocks}
\cite{CaMe} Let $P$ and $Q$ be matrices of suitable size.
Then
\begin{enumerate}[(a)]
\item $\begin{bmatrix}
        0 & P \\
        0 & Q\\
      \end{bmatrix}^\dag=
       \begin{bmatrix}
         0 & 0 \\
   (P^*P+Q^*Q)^\dag P^* & (P^*P+Q^*Q)^\dag Q^*
       \end{bmatrix}.$
\item $\begin{bmatrix}
        0 & 0 \\
        P & Q
      \end{bmatrix}^\dag= \begin{bmatrix}
                             0 & P^*(PP^*+QQ^*)^\dag \\
                             0 & Q^*(PP^*+QQ^*)^\dag
                           \end{bmatrix}.$
\end{enumerate}
\end{lemma}

Using the definition of Moore-Penrose inverse, direct calculations lead to the following representation for the Moore-Penrose of a rectangular matrix  for which its matrix block representation is block (upper) triangular with some diagonal block being nonsingular.

\begin{lemma}\label{MP triangular}
Let
$A=U\begin{bmatrix}
        A_1 & A_2 \\
        0 & A_3 \\
      \end{bmatrix}V^* \in \Cm$  such that  $A_1 \in \C^{t\times t}$ is nonsingular and $U \in {\mathbb C}^{m \times m}$ and $V \in {\mathbb C}^{n \times n}$ are unitary. Then
\begin{equation} \label{MP of A}
A^\dag = V\left[\begin{array}{cc}
A_1^*\Delta_A & -A_1^*\Delta_A A_2 A_3^\dagger\\
\Omega_A^* \Delta_A & A_3^\dagger-\Omega_A^*\Delta_A A_2 A_3^\dag
\end{array}\right]U^*,
\end{equation}
where $\Delta_A:=(A_1 A_1^*+ \Omega_A \Omega_A^*)^{-1}$ and $\Omega_A:=A_2\overline{Q}_{A_3}$.
\end{lemma}
From \eqref{MP of A}, the following expressions can be easily computed for the orthogonal projectors
\begin{equation} \label{PA and QA}
P_A= U\left[\begin{array}{cc}
I_t & 0\\
0 & P_{A_3}
\end{array}\right]U^* ~~\text{and}~~ Q_A=V\left[\begin{array}{cc}
A_1^* \Delta_A A_1 & A_1^* \Delta_A \Omega_A\\
\Omega_A^*\Delta_A  A_1 & Q_{A_3}+ \Omega_A^* \Delta_A \Omega_A
\end{array}\right]V^*.
\end{equation}

\begin{lemma}\cite[Core-EP Decomposition]{Wang} \label{CEPD} Let  $A\in \Cnn$, $\ind(A)=k$, and $\rk(A^k)=t$. Then there exist two matrices $A_1$ and $A_2$ such that $A=A_1+A_2$, where $A_1\in \cmbis$, $A_2^k=0$, $A_1^*A_2=A_2A_1=0$. Further, 
there exists a unitary matrix $U\in \Cnn$ such that
\begin{equation} \label{core EP decomposition}
A=A_1+A_2, \quad A_1 := U\left[\begin{array}{cc}
T & S \\
0 & 0
\end{array}\right]U^*, \quad
A_2 := U\left[\begin{array}{cc}
0 & 0 \\
0 & N
\end{array}\right]U^*,
\end{equation}
where $T\in \C^{t\times t}$ is nonsingular,  and $N\in \C^{(n-t)\times (n-t)}$ is nilpotent of index $k$.
\end{lemma}

In \cite{Wang} it was proved that if  $A$ is written as in \eqref{core EP decomposition}, then $A\in \cmbis$ if and only if $N=0$, in which case 
\begin{equation} \label{HS}
 A = U\left[\begin{array}{cc}
T & S \\ 0 & 0 \end{array}\right]U^*,
\end{equation}
and therefore 
\begin{equation} \label{canonical form CEP}
 A^{\core} = U\left[\begin{array}{cc}
T^{-1} & 0 \\ 0 & 0 \end{array}\right]U^*.
\end{equation}

\begin{lemma}\cite{Xu}\label{core triangular 2} Suppose $A=\begin{bmatrix} P & 0\\ Q & R \end{bmatrix}\in \Cnn$ and $P\in \Crr$. Then $A^{\core}$
exists and is in upper triangular block form if and only if $P^{\core}$  and $R^{\core}$ exist and $(I_{n-r}-RR^{\core})Q=0$. In this case,
\[A^{\core}=\begin{bmatrix} P^{\core} & 0\\ -R^{\core}QP^{\core} & R^{\core} \end{bmatrix}.\]
\end{lemma} 

\begin{lemma}\cite{Xu}\label{core triangular} Suppose $A=\begin{bmatrix} P & Q\\ 0 & R \end{bmatrix}\in \Cnn$ and $P\in \Crr$. Then $A^{\core}$
exists and is in upper triangular block form if and only if $P^{\core}$  and $R^{\core}$ exist and $(I_r-PP^{\core})Q=0$. In this case,
\[A^{\core}=\begin{bmatrix} P^{\core} & -P^{\core}QR^{\core}\\ 0 & R^{\core} \end{bmatrix}.\]
\end{lemma}

\section{One-sided $*$-orthogonality}

In this section we develop canonical forms of the  right (left) $*$-orthogonal matrices and using these  canonical forms we give characterizations of parallel sums. Also, some new results on the Moore-Penrose inverse of the sum of two matrices are presented. 

\subsection{Left and right $*$-orthogonal matrices and canonical forms}

Recall that for any  $A,B \in \Cm$, the $*$-orthogonality \cite{Hes} can be defined as
\begin{equation} \label{star orthogonality}
A\perp_{*} B  \quad \Leftrightarrow \quad A^*B=0 \quad \text{and} \quad BA^*=0.
\end{equation}

What happens if one considers only one of the two condition. It turns out to be an interesting notion of one sided $*$-orthogonality. We formally introduce the following definition:

\begin{definition}\label{left and right *orthogonal} Let $A,B \in \Cm$. It is said that:
\begin{enumerate}[\rm (a)]
\item $A$ is left  $*$-orthogonal to $B$ and denoted by $A\perp_{*,l} B$ if  $A^*B=0$.
\item $A$ is right  $*$-orthogonal to $B$ and denoted by $A\perp_{*,r} B$ if  $BA^*=0$.
\end{enumerate}
\end{definition}

\begin{remark}{\rm Note that $A\perp_{*,l} B$ if and only if $B\perp_{*,l} A$, that is,  the concept of left $*$-orthogonality possesses a symmetry property. Similarly, the right $*$-orthogonality is symmetric.}
\end{remark}

Some trivial consequences of Definition \ref{left and right *orthogonal} and \eqref{star orthogonality} is the proposition below.

\begin{proposition}\label{trivial 1} Let $A,B\in \Cm$. Then the following hold:
\begin{enumerate}[(i)]
\item $A\perp_{*,l}B$ if and only if $B^* \perp_{*,r} A^*$.
\item $A\perp_{*}B$ if and only if  $A\perp_{*,l}B$ and  $A\perp_{*,r}B$.
\end{enumerate}
\end{proposition}

It is well known that the Moore-Penrose inverse satisfies the following properties:
\begin{equation}\label{range kernel MP}
\Ra(A^{\dag})= \Ra(A^*)~~\text{and}~~ \Nu(A^{\dag})=\Nu(A^*).
\end{equation}
In view of \eqref{range kernel MP}  and the well-know equivalence $AB=0 \Leftrightarrow \Ra(B)\subseteq \Nu(A)$, it is seen that the left $*$-orthogonality can be rewritten as $A^\dag B=0$, whereas  the right $*$-orthogonality is equivalent to $BA^\dag=0$. In consequence, one can obtain an alternative definition of the left (resp., right) $*$-orthogonality between $A,B\in \Cm$, namely
\begin{eqnarray*} 
A\perp_{*,l} B & \Leftrightarrow & A^\dag B=0  ~~\Leftrightarrow~~ B^\dag A=0, \\
A\perp_{*,r} B & \Leftrightarrow & B A^\dag=0 ~~\Leftrightarrow~~  A B^\dag =0.
\end{eqnarray*}

\begin{theorem} Let $A,B \in \Cm$. Then $A$ is left (resp. right)  $*$-orthogonal to $B$ if and only if $\Ra(A)\perp \Ra(B)$ (resp. $\Ra(A^*)\perp \Ra(B^*)$).
\end{theorem}

\begin{proof}
Note that $A\perp_{*,l} B$ if  $A^*B=0$, which is equivalent to $\Ra(B)\subseteq \Nu(A^*)=\Ra(A)^\perp$. Thus, $\Ra(A)\perp \Ra(B)$. 
Similarly, $A\perp_{*,r} B$ holds if and only if $\Ra(A^*)\perp \Ra(B^*)$.
\end{proof}

Now, a simultaneous form of a pair of  left (resp., right) $*$-orthogonal matrices is developed. The main tool is the Singular Value Decomposition (SVD).

\begin{theorem}\label{svd} Let $A\in \Cm$ be a  matrix of rank $r>0$ and let $\sigma_1\ge\sigma_2\ge\cdots \ge \sigma_r>0$ be the singular values of $A$.  Then there exist unitary matrices $U\in \Cmm$ and $V\in \Cnn$ such that
\begin{equation}\label{svd A}
A=U\begin{bmatrix}                                                                              \Sigma & 0 \\
0 & 0
\end{bmatrix}V^*,
\end{equation}
where $\Sigma=diag(\sigma_1, \sigma_2, \hdots, \sigma_r)$. In particular, the Moore-Penrose inverse of $A$ is given by
\begin{equation}\label{MP respect SVD}
A^\dag=V\begin{bmatrix}                                                                              \Sigma^{-1} & 0 \\
0 & 0
\end{bmatrix}U^*.
\end{equation}
\end{theorem}

The following result provides a canonical form of two right $*$-orthogonal matrices.

\begin{theorem}\label{right canonical form} Let $A, B\in \Cm$ be such that $r=\rank(A)$. Then the following are equivalent:
\begin{enumerate}[(i)]
\item $A\perp_{*,r} B$;
\item There exist unitary matrices $U\in \Cmm$ and $V\in \Cnn$, and a nonsingular  matrix $\Sigma$ such that
\begin{equation}\label{A and B right}
 A=U\begin{bmatrix}
 \Sigma & 0 \\
  0 & 0
 \end{bmatrix}V^*, \quad B= U\begin{bmatrix}
               0 & B_2 \\
               0 & B_4
              \end{bmatrix} V^*,
\end{equation}
where $B_2\in \C^{r\times (n-r)}$ and $B_4\in \C^{(m-r)\times (n-r)}$.
\end{enumerate}
In this case,
\begin{equation}\label{right B+}
 B^\dag = V\begin{bmatrix}
       0 & 0 \\
(B_2^*B_2+B_4^*B_4)^\dag B_2^* & (B_2^*B_2+B_4^*B_4)^\dag B_4^*
\end{bmatrix}U^*.
\end{equation}
\end{theorem}

\begin{proof}
(i) $\Rightarrow$ (ii). By Theorem \ref{svd}, it is clear that  $A$ can be written as in \eqref{svd A}.
Let $B$ be partitioned as
$B=U \begin{bmatrix}
 B_1 & B_2 \\
 B_3 & B_4
 \end{bmatrix}V^*$,
in conformation with the partition of $A$.  \\
Since  $A\perp_{*,r} B$, i.e., $BA^*=0$, direct calculations show that  $B_1\Sigma^*=0$ and $B_3\Sigma^*=0$. So, the nonsingularity of $\Sigma$ implies $B_1=0$ and $B_3=0$. \\
(ii) $\Rightarrow$ (i). By using \eqref{A and B right}, it is easy to check that $BA^*=0$, i.e., $A\perp_{*,r} B$. \\
Finally,  from \eqref{A and B right} and Lemma \ref{MP blocks}, it follows that \eqref{right B+} holds.
\end{proof}

A similar result for left $*$-orthogonal matrices holds.

\begin{theorem}\label{left canonical form}
Let $A, B\in \Cm$ be such that $r=\rank(A)$. Then the following are equivalent:
\begin{enumerate}[(i)]
\item $A\perp_{*,l} B$;
\item There exist unitary matrices $U\in \Cmm$ and $V\in \Cnn$, and a nonsingular  matrix $\Sigma$ such that
\begin{equation}\label{B respect to A.}
 A=U\begin{bmatrix}
 \Sigma & 0 \\
  0 & 0
 \end{bmatrix}V^*, \quad B=U\begin{bmatrix}
               0 & 0 \\
               B_3 & B_4
              \end{bmatrix} V^*,
\end{equation}
where $B_3\in \C^{(m-r)\times r}$ and $B_4\in \C^{(m-r)\times (n-r)}$.
\end{enumerate}
In this case,
\begin{equation*}\label{left B+}
 B^\dag = V \begin{bmatrix}
                         0 & B_3^*(B_3B_3^*+B_4B_4^*)^\dag  \\
      0 & B_4^*(B_3B_3^*+B_4B_4^*)^\dag
              \end{bmatrix} U^*.
\end{equation*}
\end{theorem}

\subsection{Left and right $*$-orthogonality and parallel sum}

The earliest most definition of parallel summable matrices is for Hermitian semidefinite matrices \cite{AnDu}, but we choose to define it for rectangular matrices from usage point of view.

\begin{definition}\label{definition PS} Let $A,B\in \Cm$ and at least one of $A$ and $B$ are non-null. Then  $A$ and $B$ are said to be parallel summable (or, in short, p.s.) if $A(A+B)^{-}B$ is invariant under any choice of  $(A+B)^{-}$.
When this is so, the common value of $A(A+B)^{-}B$ is called the parallel sum of $A$ and $B$, and  denoted by $A:B$.
\end{definition}

\begin{remark}\label{remark PS1}{\rm Note that if $A$ and $B$ are parallel summable then
 $A:B=A(A+B)^\dag B.$ }
\end{remark}

A well-known characterization of parallel summable matrices is established  below.

\begin{lemma}\label{P2} \cite{MiOd, MiBhMa} Let $A, B \in \Cm$.  Then the following are equivalent:
\begin{enumerate}[{\rm(i)}]
\item $A$ and $B$ are parallel summable;
\item $\Ra(A)\subseteq \Ra(A+B)$ and $\Ra(A^*) \subseteq \Ra(A^*+B^*)$.
\end{enumerate}
\end{lemma}
\begin{proposition}\label{righ parallel summable} Let $A, B\in \Cm$ be two matrices written as in \eqref{A and B right}. Then $A$ is parallel summable to $B$ if and only if $\Ra(B_2^*)\subseteq \Ra(B_4^*)$.
\end{proposition}
\begin{proof}
Let us consider  decompositions of $A$ and $B$ given in \eqref{A and B right}, that is,
\begin{equation*}\label{A+B}
A+B =  U\begin{bmatrix}
 \Sigma & B_2\\
  0 & B_4
 \end{bmatrix}V^*.
\end{equation*}
From Lemma \ref{MP triangular} it follows
\begin{equation} \label{MP of A+B}
(A+B)^\dag = V\left[\begin{array}{cc}
\Sigma^*\Delta_{A+B} & -\Sigma^*\Delta_{A+B} B_2 B_4^\dagger\\
\Omega_{A+B}^*\Delta_{A+B} & B_4^\dagger-\Omega_{A+B}^*\Delta_{A+B} B_2 B_4^\dagger
\end{array}\right]U^*,
\end{equation}
where $\Delta_{A+B}=(\Sigma \Sigma^*+ \Omega_{A+B} \Omega_{A+B}^*)^{-1}$ and $\Omega_{A+B}=B_2\overline{Q}_{B_4}$. \\
Now,  \eqref{MP of A+B} and \eqref{PA and QA}  imply
 \begin{eqnarray}
 P_{A+B}&=& V\begin{bmatrix}
 I_r & 0 \\
  0 & P_{B_4}
 \end{bmatrix}V^*, \label{P_A+B} \\
Q_{A+B} &=& U\left[\begin{array}{cc}
\Sigma^* \Delta_{A+B} \Sigma & \Sigma^* \Delta_{A+B} \Omega_{A+B}\\
\Omega_{A+B}^*\Delta_{A+B}  \Sigma & Q_{B_4}+ \Omega_{A+B}^* \Delta_{A+B} \Omega_{A+B}
\end{array}\right]U^*. \label{Q_A+B}
\end{eqnarray}
By Lemma \ref{P2},  $A$ is parallel summable to $B$ if and only if $\Ra(A)\subseteq \Ra(A+B)$ and $\Ra(A^*)\subseteq \Ra(A^*+B^*)$, which in turn is equivalent to 
\begin{equation} \label{equivalence parallel sum}
P_{A+B}A=A \quad \text{and} \quad AQ_{A+B}=A.
\end{equation}
From \eqref{P_A+B}, it is easy to see that the first equation in \eqref{equivalence parallel sum} always is true. So, $A$ is parallel summable to $B$ if and only if $AQ_{A+B}=A$. In consequence,   from \eqref{Q_A+B},  direct calculations show that  
\begin{equation}\label{eq2}
A ~~\text{is p.s.} ~~B ~~ \Leftrightarrow~~\left[\begin{array}{cc}
\Sigma \Sigma^* \Delta_{A+B} \Sigma & \Sigma \Sigma^* \Delta_{A+B} \Omega_{A+B}\\
0 & 0
\end{array}\right]
=\begin{bmatrix}
 \Sigma & 0 \\
  0 & 0
 \end{bmatrix}.
\end{equation}
Note that \eqref{eq2} holds if and only if  $\Sigma \Sigma^* \Delta_{A+B} \Sigma =\Sigma$ and
$\Sigma \Sigma^* \Delta_{A+B} \Omega_{A+B}=0$. By using the nonsingularity of $\Sigma$ and $\Delta_{A+B}$,  it is clear that the two last equations are equivalent to   $\Omega_{A+B}=0$. Therefore, $B_2 \overline{Q}_{B_4}=0$  or equivalently  $\Ra(I_{n-r}-Q_{B_4})\subseteq \Nu(B_2)$ which in turn is equivalent to $\Nu(Q_{B_4})\subseteq \Nu(B_2)$. So, as $\Nu(Q_{B_4})=\Nu(B_4)=\Ra(B_4^*)^{\perp}$ and $\Nu(B_2)=\Ra(B_2^*)^{\perp}$ we have that $A$ is parallel summable to $B$ if and only if  $\Ra(B_2^*)\subseteq \Ra(B_4^*)$.
\end{proof}
Next, we stablish  the main result of this section. 

\begin{theorem}\label{corollary right 1} Let $A, B\in \Cm$ be such that $A\perp_{*,r} B$. Then $A$ is parallel summable to $B$ if and only if $\Ra(\overline{Q}_AB^*A)\subseteq \Ra(\overline{Q}_A B^*\overline{P}_A)$. In this case, $A:B=0$.
\end{theorem}
\begin{proof}
Since $A\perp_{*,r} B$, Theorem \ref{right canonical form} implies  that $A$ and $B$ can be written as in
\eqref{A and B right}. Therefore, by Proposition \ref{righ parallel summable}  it is sufficient to prove that $\Ra(\overline{Q}_AB^*A)=\Ra(\overline{Q}_AB^*P_A)\subseteq \Ra(\overline{Q}_A B^*\overline{P}_A)$ is equivalent to $\Ra(B_2^*)\subseteq \Ra(B_4^*)$. In fact,
\begin{eqnarray}
\overline{Q}_AB^*P_A &=& V\begin{bmatrix}
 0 & 0 \\
  0 & I_{n-r}
 \end{bmatrix}\begin{bmatrix}
 0 & 0 \\
  B_2^* & B_4^*
 \end{bmatrix}\begin{bmatrix}
 I_r & 0 \\
  0 & 0
 \end{bmatrix}U^*=V\begin{bmatrix}
 0 & 0 \\
  B_2^* & 0
 \end{bmatrix}U^*, \label{corollary eq1} \\
 \overline{Q}_AB^*\overline{P}_A &=& V\begin{bmatrix}
 0 & 0 \\
  0 & I_{n-r}
 \end{bmatrix}\begin{bmatrix}
 0 & 0 \\
  B_2^* & B_4^*
 \end{bmatrix}\begin{bmatrix}
 0 & 0 \\
  0 & I_{n-r}
 \end{bmatrix}U^*=V\begin{bmatrix}
 0 & 0 \\
 0 &  B_4^*
 \end{bmatrix}U^*. \label{corollary eq2}
\end{eqnarray}
From \eqref{corollary eq1} and \eqref{corollary eq2} the affirmation is clear. \\
In order to prove $A:B=0$, we assume that $A$ is parallel summable to $B$. Therefore $A:B=A (A+B)^\dag B$. Moreover, as in the proof of Proposition \ref{righ parallel summable} we have that $\Ra(B_2^*)\subseteq \Ra(B_4^*)$ is equivalent to $B_2(I_{n-r}-Q_{B_4})=0$. So, $(I_{n-r}-Q_{B_4})B_2^*=0$ and
$\Delta_{A+B}= (\Sigma^*)^{-1} \Sigma^{-1}$. Now, from \eqref{A and B right} and \eqref{MP of A+B} we obtain
\begin{equation*}
A(A+B)^\dag B = U\begin{bmatrix}
 \Sigma & 0 \\
  0 & 0
 \end{bmatrix}\left[\begin{array}{cc}
\Sigma^{-1} & -\Sigma^{-1} B_2 B_4^\dag\\
0 & B_4^\dag
\end{array}\right]\begin{bmatrix}
 0 & B_2 \\
  0 & B_4
 \end{bmatrix}V^*=U\begin{bmatrix}
 0 & B_2(I_{n-r}-Q_{B_4}) \\
  0 & 0
 \end{bmatrix}V^*=0,
\end{equation*}
that is, $A:B=0$. This, completes the proof.
\end{proof}

\begin{proposition} \label{left parallel summable} Let $A, B\in \Cm$ be two matrices of the forms \eqref{B respect to A.}. Then $A$ is parallel summable to $B$ if and only if $\Ra(B_3)\subseteq \Ra(B_4)$.
\end{proposition}

\begin{theorem} Let $A, B\in \Cm$ be such that $A\perp_{*,l} B$. Then $A$ is parallel summable to $B$ if and only if 
$\Ra(\overline{P}_ABA^*)\subseteq \Ra(\overline{P}_AB\overline{Q}_A)$. In this case, $A:B=0$.
\end{theorem}

\subsection{Additivity of the Moore-Penrose inverse}

This  section deals with  the additivity property of the Moore-Penrose inverse under conditions of right (resp., left) $*$-orthogonality.

Recall that a reflexive and transitive binary relation on a non-empty set is called a pre-order. A pre-order is a partial order if it is antisymmetric. We further note that the generalized inverses play an important role in the study of matrix partial orders \cite{MiBhMa}.
The star  order was defined by Drazin (1978)   as follows:
\begin{equation*}\label{st}
\text{For}~ A,B\in \Cm:~ A \st B ~\Leftrightarrow~  A^*A= A^*B ~\text{and}~ AA^*=BA^*.
\end{equation*}
Hartwig and Styan  \cite{HaSt} gave the following characterization of the star partial order:

  \begin{equation}\label{star characterization}
  A\st B ~\Leftrightarrow~\rk(B-A)=\rk(B)-\rk(A)~\text{and}~ (B-A)^\dag=B^\dag-A^\dag.
  \end{equation}
  The condition $\rk(B-A)=\rk(B)-\rk(A)$ is known as the \emph{rank subtractivity} and the condition $(B-A)^\dag=B^\dag-A^\dag$ is called the
  \emph{dagger-subtractivity}. They also noted that if $A,B$ are  $*$-orthogonal, then \eqref{star characterization} can be rephrased as
  \begin{equation}\label{star characterization 1}
A\perp_* B ~~\Leftrightarrow~~ (A+B)^\dag=A^\dag+B^\dag \quad\text{and}\quad \rank(A+B)=\rank(A)+\rank(B),
\end{equation}
which is a useful characterization of $*$-orthogonality.

The first property given in \eqref{star characterization 1} is known as {\it dagger-additivity} whereas
the second property is called {\it rank-additivity}.

In view of \eqref{star characterization 1} it is clear that $A\perp_{*} B$ implies  dagger-additivity, that is,
\begin{equation*}\label{implies 1}
A^*B=0 ~ \text{and} ~ BA^*=0  ~\Rightarrow~ (A+B)^\dag=A^\dag+B^\dag.
\end{equation*}
Recently, Baksalary et al.  \cite{BaSiTr} considered the following alternative conditions for the Moore-Penrose inverse of the sum of $A+B$ to be additive:
\begin{equation}\label{implies 2}
AB^*+BB^*=0 \quad \text{and} \quad B^*A+B^*B=0 .
\end{equation}
However, the authors proved that the pairs of identities given in \eqref{implies 2} and the $*$-orthogonality are independent.
The authors also proved that the identities in \eqref{implies 2} are satisfied if and only if $B\st -A$ (or equivalently $-B \st A$).

Note that \eqref{implies 2} can be rewritten in term of one-sided $*$-orthogonality as follows:
\begin{equation*}\label{implies 3}
B\perp_{*,r} A+B  \quad \text{and} \quad B\perp_{*,l} A+B.
\end{equation*}
\begin{remark}
If $A^*B=0$  or $BA^*=0$ in \eqref{implies 2} then $B=0$.
\end{remark}

Therefore, a natural question that arises is under what weaker conditions the additivity of the Moore-Penrose inverse can also be true?

Our next results show weaker conditions than $*$-orthogonality to ensure the additivity property of the Moore-Penrose inverse. Before, we need an auxiliary lemma. 

\begin{lemma}\label{lemma additivity right} Let $A, B\in \Cm$ be two matrices written as in \eqref{A and B right}. Then $(A+B)^\dag=A^\dag + B^\dag$ if and only if $B_2=0$.
\end{lemma}
\begin{proof}
We consider $A$ and $B$ of the form \eqref{A and B right}. From, \eqref{MP respect SVD}, \eqref{right B+}, and \eqref{MP of A+B} we have that $(A+B)^\dag=A^\dag + B^\dag$ if and only if
\begin{equation}\label{addtivity sum eq1}
\left[\begin{array}{cc}
\Sigma^*\Delta_{A+B} & -\Sigma^*\Delta_{A+B} B_2 B_4^\dagger\\
\Omega_{A+B}^*\Delta_{A+B} & B_4^\dagger-\Omega_{A+B}^*\Delta_{A+B} B_2 B_4^\dagger
\end{array}\right]=\begin{bmatrix}
       \Sigma^{-1} & 0 \\
(B_2^*B_2+B_4^*B_4)^\dag B_2^* & (B_2^*B_2+B_4^*B_4)^\dag B_4^*
\end{bmatrix},
\end{equation}
where $\Delta_{A+B}=(\Sigma \Sigma^*+ \Omega_{A+B} \Omega_{A+B}^*)^{-1}$ and $\Omega_{A+B}=B_2\overline{Q}_{B_4}$. \\
We will prove that \eqref{addtivity sum eq1} holds if and only if $B_2=0$. In fact, as $\Sigma$ and $\Delta_{A+B}$ are nonsingular, clearly \eqref{addtivity sum eq1} is equivalent to
\begin{eqnarray}
\Sigma\Sigma^*\Delta_{A+B} &=& I_r, \label{add 1}\\
B_2B_4^\dag &=& 0, \label{add 2}\\
\Omega_{A+B}^*\Delta_{A+B}  &=& (B_2^*B_2+B_4^*B_4)^\dag B_2^*, \label{add 3} \\
B_4^\dagger-\Omega_{A+B}^*\Delta_{A+B} B_2 B_4^\dagger &=& (B_2^*B_2+B_4^*B_4)^\dag B_4^*. \label{add 4}
\end{eqnarray}
By using the fact that $\Delta_{A+B}=(\Sigma \Sigma^*+ \Omega_{A+B} \Omega_{A+B}^*)^{-1}$ , from \eqref{add 1} it follows that $\Omega_{A+B} \Omega_{A+B}^*=0$, whence $\Omega_{A+B}=0$.  In consequence, as $\Omega_{A+B}=B_2\overline{Q}_{B_4}$, from \eqref{add 2} we have  $B_2=0$. Conversely, if $B_2=0$, it easy to see that conditions \eqref{add 1}-\eqref{add 4} are true.
\end{proof}

\begin{theorem}\label{theorem additivity right} Let $A, B\in \Cm$ be such that $A\perp_{*,r} B$. Then $(A+B)^\dag=A^\dag + B^\dag$ if and only if $A^*B(I_n-Q_A)=0$.
\end{theorem}
\begin{proof}
As $A\perp_{*,r} B$, from Theorem \ref{right canonical form} we have
that $A$ and $B$ can be written of the forms \eqref{A and B right}. Thus, from  \eqref{corollary eq1} we obtain
\begin{equation*}\label{additivity main theorem}
(I_n-Q_A)B^*P_A =V\begin{bmatrix}
 0 & 0 \\
  B_2^* & 0
 \end{bmatrix}U^*.
 \end{equation*}
 Now, the result  follows from Lemma \ref{lemma additivity right} and the fact that $(I_n-Q_A)B^*A=0$ is equivalent to $(I_n-Q_A)B^*P_A=0$.
\end{proof}

One can obtain a similar result for left $*$-orthogonal matrices.

\begin{lemma}\label{lemma additivity left} Let $A, B\in \Cm$ be two matrices of the forms \eqref{B respect to A.}. Then $(A+B)^\dag=A^\dag + B^\dag$ if and only if $B_3=0$.
\end{lemma}

\begin{theorem}\label{theorem additivity left} Let $A, B\in \Cm$ be such that $A\perp_{*,l} B$. Then $(A+B)^\dag=A^\dag + B^\dag$ if and only if $(I_m-P_A)BA^*=0$.
\end{theorem}

Note that  theorems \ref{theorem additivity right} and \ref{theorem additivity left} generalize well-known results about the additivity of the Moore-Penrose inverse  of the sum of two  $*$-orthogonal matrices.
\begin{corollary} Let $A, B\in \Cm$ be such that $A\perp_{*} B$. Then $(A+B)^\dag=A^\dag + B^\dag$.
\end{corollary}

\begin{proof}
As $A\perp_{*} B$, Proposition \ref{trivial 1} implies $A\perp_{*,l} B$ and $A\perp_{*,r} B$. Moreover, $A\perp_{*,l} B$ implies $(I_n-Q_A)B^*A=0$. Thus, from Thoerem \ref{theorem additivity right} we have $(A+B)^\dag=A^\dag + B^\dag$.
\end{proof}

We finish this section providing an expression for $(A+B)^\dag$ when $A$ and $B$ are left (resp. right) $*$-orthogonal matrices with the requirement that their ranges are disjoints, that is, $\Ra(A)\cap \Ra(B)=\{0\}$.  Before we recall the following fact about  the Moore-Penrose inverse of the sum of two Hermitian positive semidefinite matrices. 

\begin{lemma}\cite[Theorem 4]{BaTr21}\label{thm 4 of BT} Let $A, B\in \Cm$ be such that  $\Ra(A)\cap \Ra(B)=\{0\}$. Then $(AA^*+BB^*)^\dag=[(\overline{P}_A B)(\overline{P}_A B)^*]^\dag+[(\overline{P}_B A)(\overline{P}_B A)^*]^\dag$.
\end{lemma}

\begin{theorem} Let $A, B\in \Cm$ be such that $A\perp_{*,r} B$ and $\Ra(A)\cap \Ra(B)=\{0\}$. Then
\[(A+B)^\dag=(A+B)^*
([(\overline{P}_A B)(\overline{P}_A B)^*]^\dag+[(\overline{P}_B A)(\overline{P}_B A)^*]^\dag).\]
\end{theorem}
\begin{proof} By definition  $A\perp_{*,r} B$ implies $BA^*=0$. In consequence,
\begin{equation*}
(A+B)^\dag= (A+B)^*[(A+B)(A+B)^*]^\dag =(A+B)^*(AA^*+BB^*)^\dag.
\end{equation*}
Thus, the conclusion follows from Lemma \ref{thm 4 of BT}.
\end{proof}
Using the identity $(A+B)^\dag= [(A+B)^*(A+B)]^\dag (A+B)^*$ , 
one can  obtain a similar result for the left $*$-orthogonality. 
\begin{theorem} Let $A, B\in \Cm$ be such that $A\perp_{*,l} B$ and $\Ra(A^*)\cap \Ra(B^*)=\{0\}$. Then
\[(A+B)^\dag=([(\overline{P}_{A^*} B^*)(\overline{P}_{A^*} B^*)^*]^\dag+[(\overline{P}_{B^*} A^*)(\overline{P}_{B^*} A^*)^*]^\dag) (A+B)^*.\]
\end{theorem}

\section{Group matrices and one-sided core-orthogonality}

We observe that any statement that holds for rectangular matrices also holds for square matrices. But how the things change when we restrict the matrices to be of index $1?$ Indeed, the matrices of index $1$ have the  qualification that they possess a unique group inverse and many things about them become easy to understand and in control when used. In this section, we intend to study the impact on one-sided $*$-orthogonality.

\subsection{Left and right core-orthogonality}

We know that $A$ is orthogonal to $B$ (in the usual sense) if $AB=BA=0$. Whereas  $A$ and $B$ are $*$-orthogonal if $A^*B=0$ and $BA^*=0$, and denoted by $A\perp_*B$. 

Recently, in \cite{FeMa2} the authors introduced the concept of core- orthogonality by using the core inverse $A^{\core}$ instead of the conjugate transpose $A^*$. 

Thus, this approach will lead to a new type of orthogonality for group matrices. 

 \begin{definition}\cite{FeMa2} \label{def PCO} Let $A,B\in \cmbis$. Then $A$ is core-orthogonal to $B$ (denoted by $A\perp_{\core} B$) when
\begin{equation*} \label{pco1}
 A^{\core}B=0\quad \text{and}\quad BA^{\core}=0.
\end{equation*}
\end{definition}

In \cite{FeMa2}, by using the interesting feature of the core inverse 
$\Ra(A^{\core})=\Ra(A)$ and $\Nu(A^{\core})=\Nu(A^*)$
it was proved that Definition \ref{def PCO} can be rewritten as 

\begin{equation}\label{alternative def}
A\perp_{\core} B ~\Leftrightarrow~ A^*B=0 ~~\text{and}~~ BA=0.
\end{equation} 
So, the core-orthogonality may be regarded as a orthogonality between the usual orthogonality   and the $*$-orthogonality for group matrices.
Moreover, by using Definition \ref{left and right *orthogonal} we have 
\begin{equation*}\label{alternative def 2}
A\perp_{\core} B ~\Leftrightarrow~ A\perp_{*,l} B ~~\text{and}~~ A^*\perp_{*,r} B.
\end{equation*} 
Thus, it is clear that $A\perp_{*,l} B$ does not necessarily imply  $A$ core-orthogonal to $B$. Where as $A$  core-orthogonal to $B$ implies $A\perp_{*,l} B.$ We give the following example:
 \begin{example} Let $A= \begin{bmatrix}
                          1 & 1 \\
                          0 & 0 
                        \end{bmatrix},~ B=\begin{bmatrix}
                                            0 & 0 \\
                                            1 & 0 
                                          \end{bmatrix}.$ 
Then $A^*B=0, ~A^{\core}B=0$ and $BA^{\core}\neq 0.$
 \end{example}
 Similar deductions are possible for $A\perp_{*,r} B$.
 
We next give a new definition in view of above discussion:

 \begin{definition}\label{left and right core-orthogonal} Let $A,B\in \cmbis$. We say that
\begin{enumerate}[\rm (a)]
\item $A$ is left core-orthogonal to $B$ and denoted by $A\perp_{\tiny\core,l} B$ if  $A^{\core}B=0$.
\item $A$ is right  core-orthogonal to $B$ and denoted by $A\perp_{\core,r} B$ if  $BA^{\core}=0$.
\end{enumerate}
\end{definition}
 Note that for  $A,B\in \cmbis$ we have
\begin{equation}\label{equivalence core orthogonal}
 A\perp_{{\core}} B ~\Leftrightarrow~
 A\perp_{{\core},l} B ~ \text{and} ~ A\perp_{{\core},r} B.
\end{equation}

The previous observations lead to the following result.
 \begin{theorem} \label{sb 1}Let $A,B\in \cmbis.$ Then the following are equivalent:
 \begin{enumerate}[(i)]
\item $A\perp_{\core,l} B$ (resp. $A\perp_{\core,r} B$ );
\item $A\perp_{*,l} B$ (resp. $A^*\perp_{*,r} B$);
\item $B\perp_{\core,l} A$ (resp. $B \perp_{\core,r} A^*$).
\end{enumerate}
\end{theorem}

\begin{remark}\label{remark symmetry} Note that the left core-orthogonality has the  symmetry property  while the right core-orthogonality is not.  
\end{remark}

A simultaneous form of a pair of  left (resp., right) core-orthogonal matrices is developed below using the core-EP decomposition. 

\begin{theorem}\label{sb 2} Let $A,B\in \cmbis$ and $r=\rk(A)$. Then the following are equivalent:
 \begin{enumerate}[(i)]
\item $A\perp_{\core,l} B$;
\item There exists a unitary matrix $U\in \Cnn$ and a nonsingular  matrix $T\in \Crr$  such that
\begin{equation}\label{B respect to A}
 A=U\begin{bmatrix}
 T& S \\
  0 & 0
 \end{bmatrix}U^*, \quad B=U\begin{bmatrix}
               0 & 0 \\
               B_3 & B_4
              \end{bmatrix} U^*,
\end{equation}
where $(I_{n-r}-B_4B_4^{\core})B_3=0$ and $B_4\in \C^{\cm}_{n-r}$.
\end{enumerate}
Further,
\begin{equation}\label{bcore}
 B^{\core}=U\begin{bmatrix}
                      0 & 0 \\
                      0 & B^{\core}_4
                    \end{bmatrix}U^*.
 \end{equation}
 \end{theorem}
 
\begin{proof} We consider $A$ as in \eqref{HS}.
Partitioning $B$ according to partition of $A$, we have 
$B= U\begin{bmatrix}
    B_1 & B_2\\
    B_3 & B_4 
  \end{bmatrix}U^*$,
where $B_1\in \C^{r\times r}$. From Definition \ref{left and right core-orthogonal} and \eqref{canonical form CEP}, a straightforward computation shows that
\[A^{\core}B=0 \Leftrightarrow B= 
U\begin{bmatrix}
    0 & 0\\
    B_3 & B_4 
  \end{bmatrix}U^*.\]
Since $B\in \cmbis$, it is clear that $B^{\core}$ exists. Thus, by Lemma \ref{core triangular 2} we have that $B_4^{\core}$ exists and $(I_{n-r}-B_4B_4^{\core})B_3=0$.  Moreover,  it follows \eqref{bcore}.
\end{proof}

A similar result for right core-orthogonal matrices can be obtained by using Lemma \ref{core triangular} instead of Lemma \ref{core triangular 2}.

\begin{theorem}\label{sb 2 right} Let $A,B\in \cmbis$ and $r=\rk(A)$. Then the following are equivalent:
 \begin{enumerate}[(i)]
\item $A\perp_{\core,r} B$;
\item There exists a unitary matrix $U\in \Cnn$  and a nonsingular  matrix $T\in \Crr$  such that
\begin{equation}\label{A respect B right}
 A=U\begin{bmatrix}
 T& S \\
  0 & 0
 \end{bmatrix}U^*, \quad B=U\begin{bmatrix}
               0 & 0\\
               0 & B_4
              \end{bmatrix} U^*,
\end{equation}
where  $B_4\in \C^{\cm}_{n-r}$.
\end{enumerate}
Further,
\begin{equation*} 
 B^{\core}=U\begin{bmatrix}
                      0 & 0 \\
                      0 & B^{\core}_4
                    \end{bmatrix}U^*.
 \end{equation*}
 \end{theorem}

\subsection{Additivity of the core inverse of the sum}

In this section, we explore the conditions under which the core inverse of the sum of two matrices is the sum of their respective core inverses. 

Ferreyra and Malik \cite{FeMa2} noted that the relation of core orthogonality lacks in symmetry that both $*$-orthogonality and usual orthogonality possess.
In order to recover the property of symmetry the following  concept was defined.

\begin{definition}\cite{FeMa2} \label{def strongly} Let $A,B\in \cmbis$. Then $A$ is strongly core-orthogonal to $B$ (denoted by $A\perp_{\core,S} B$) when $A\perp_{\core} B $ and $B\perp_{\core} A$.
\end{definition}
In \cite{FeMa2} it was proved that 
\begin{equation}\label{strongly eq}
A\perp_{\core,S} B ~~\Leftrightarrow~~ A^{\core}B=0, ~~ BA^{\core}=0,~~\text{and}~~ AB^{\core}=0.
\end{equation}

In view of Definition \ref{left and right core-orthogonal}  we have that \eqref{strongly eq} can be rewritten as

\begin{equation*}\label{strongly characterization 2}
 A\perp_{\core,S} B ~\Leftrightarrow~
 A\perp_{{\core},l} B,~~A\perp_{{\core},r} B,~~\text{and}~~ B\perp_{{\core},r} A.
\end{equation*}   

The authors also proved that the strongly core-orthogonality is a sufficient condition for the core-additivity, that is,
\begin{equation*}\label{strongly characterization}
A\perp_{\core,S} B ~~\Rightarrow ~~ (A+B)^{\core}=A^{\core}+B^{\core}.
\end{equation*} 
However, the strongly core-orthogonality is not a necessary condition. 
\begin{example} Let $A=\begin{bmatrix}
                   1 & 0 \\
                   0 & 1 
                 \end{bmatrix}$ and $B=\begin{bmatrix}
                                    -1 & 0 \\
                                    0 & 0
                                  \end{bmatrix}$. 
It is easy to see that $A^{\core}=A$, $B^{\core}=B$, and $(A+B)^{\core}=A+B$. Thus, $(A+B)^{\core}=A^{\core}+B^{\core}$. However, $A^{\core}B\neq 0$, $BA^{\core}\neq 0$ and $AB^{\core}\neq 0$. 
\end{example}

Our target here is prove under what other conditions, we have  $(A+B)^{\core} =A^{\core}+B^{\core}$. 

\begin{theorem}\label{sb 3} Let $A,B\in \cmbis$ be such that  $A\perp_{\core,l} B$. Then $(A+B)^{\core} =A^{\core}+B^{\core}$ if and only $A\perp_{\core,r} B$ and $B\perp_{\core,r} A$.
\end{theorem}
\begin{proof} 
$\Rightarrow)$ Since $A\perp_{\core,l} B$, by Theorem \ref{sb 2} we can consider $A$ and $B$  as in \eqref{B respect to A}. \\ Thus, from \eqref{canonical form CEP} and \eqref{bcore} we obtain 
\begin{equation}\label{A+B core 1}
(A+B)^{\core} =A^{\core}+B^{\core}=
U\begin{bmatrix}
 T^{-1}& 0 \\
  0 & B_4^{\core}
 \end{bmatrix}U^*.
\end{equation}
It is well-known that  $(A+B)((A+B)^{\core})^2=(A+B)^{\core}$. Thus, direct calculations yields to 
\begin{equation}\label{A+B core 2}
(A+B)^{\core}=(A+B)((A+B)^{\core})^2=U\begin{bmatrix}
 T^{-1}& S(B_4^{\core})^2 \\
  B_3T^{-2} & B_4^{\core}
 \end{bmatrix}U^*.
\end{equation}
From \eqref{A+B core 1} and \eqref{A+B core 2} we have $B_3=0$ and $S(B_4^{\core})^2=0$.  Note that the second equality is equivalent to $SB_4^{\core}=0$ because  $(B_4^{\core})^2 B_4=B^{\core}_4$. In consequence, from \eqref{canonical form CEP} and \eqref{bcore} it is easy to check that $BA^{\core}=0$ and $AB^{\core}=0$, i.e., $A\perp_{\core,r} B$ and $B\perp_{\core,r} A$.  \\
$\Leftarrow)$ Let $A\perp_{\core,r} B$. By Theorem \ref{sb 2 right} we can consider $A$ and $B$  as in \eqref{A respect B right}. Thus, Lemma \ref{core triangular} implies 
\begin{equation}\label{core of A+B}
(A+B)^{\core}=\begin{bmatrix} 
T^{-1} & -T^{-1}SB_4^{\core}\\ 
0 & B_4^{\core} 
\end{bmatrix}.
\end{equation}
Now, by hypothesis also we have $B\perp_{\core,r} A$, i.e., $AB^{\core}=0$. Thus, 
\begin{equation}\label{SB4 core}
0=AB^{\core} =
U\begin{bmatrix}
    0 & SB_4^{\core}\\
    0 & 0 
  \end{bmatrix}U^*~\Rightarrow~ SB_4^{\core}=0.
\end{equation}
From \eqref{core of A+B} and \eqref{SB4 core} we obtain 
\[
(A+B)^{\core}=\begin{bmatrix} 
T^{-1} & 0\\ 
0 & B_4^{\core} 
\end{bmatrix}=U\begin{bmatrix} 
T^{-1} & 0\\ 
0 & 0 
\end{bmatrix}U^*+U\begin{bmatrix} 
0& 0\\ 
0 & B_4^{\core} 
\end{bmatrix}U^*=A^{\core}+B^{\core}.
\]
\end{proof}
 
Let us now consider $A,B\in \cmbis$ such that

\begin{equation}\label{N1}
A^{\core}B + A^{\core}A=0 ~~ \text{and}~~ BA^{\core}+AA^{\core}=0.
\end{equation}

Note that \eqref{N1}  can be rewritten in term of one-sided core-orthogonality as follows:

\begin{equation}\label{N1 and N2 core 1}
A\perp_{\core,l} A+B \quad \text{and} \quad A\perp_{\core,r} A+B.
\end{equation}

Thus, from  \eqref{alternative def} and \eqref{equivalence core orthogonal} we have that \eqref{N1 and N2 core 1} is equivalent to 
\begin{equation}\label{alternative N1 and N2 core 1}
(A+B)A=0~~\text{and}~~A^*(A+B)=0.
\end{equation}

\begin{theorem} \label{sb 7} Let $A,B\in \cmbis$ such that the conditions in \eqref{N1} are satisfied. Then $\Ra(B)\subseteq \Ra(A)$  and $\Ra(A^*) \subseteq \Ra(B^*)$.
\end{theorem}
\begin{proof} 
Let $A$ and $B$ be two matrices satisfying \eqref{N1}. In consequence, by  \eqref{alternative N1 and N2 core 1} we know that $(A+B)A=0$ and $A^*(A+B)=0$. Now, we consider  $x \in \Nu(B)$. By using the second condition we obtain $A^*Ax=0$, whence $x\in \Nu(A)$. Thus, $\Nu(B)\subseteq \Nu(A)$ which is equivalent to $\Ra(A^*) \subseteq \Ra(B^*)$. Similarly, the condition $(A+B)A=0$ implies $\Nu(A^*)\subseteq \Nu(B^*)$ or equivalently $\Ra(B) \subseteq \Ra(A)$.
\end{proof}

Another interesting observation  is that the identities in \eqref{N1} have a relevant interpretation dealing with the core partial order introduced in \cite{BaTr}. Recall that for $A,B\in \cmbis$, the core partial order is defined as  
\[A\co B ~\Leftrightarrow~ A^{\core} A= A^{\core} B ~~\text{and}~~ AA^{\core}=BA^{\core}.\]
In the light of the above, we conclude 
\begin{equation}\label{N1 equivalent 1}
A^{\core}B + A^{\core}A=0  ~~ \text{and}~~ BA^{\core}+AA^{\core}=0~~  \Leftrightarrow~~ A\co -B. 
\end{equation}

Recently, in \cite{FeMa1, FeMa2}  the authors obtained several properties and new characterizations of the core partial order.  Some  of these results asserts that, for any $A,B\in \cmbis$ such that $A \co B$ we have
\begin{equation} \label{core subtractivity}
 AB=BA \Leftrightarrow A^2=AB \Leftrightarrow A^2\co B^2 \Leftrightarrow B-A\co B \Leftrightarrow (B-A)^{\core}=B^{\core}-A^{\core}.
\end{equation}

\begin{theorem}\label{sb 11} Let $A,B\in \cmbis$ such that the conditions in \eqref{N1} are satisfied.  Then the following conditions are equivalent: 
\begin{enumerate}[(i)]
\item $AB=BA$;
\item $A^2=-AB$;
\item $A^2\co B^2$;
\item $A+B ~\co~ B$;
\item $(A+B)^{\core}= A^{\core}+B^{\core}$.
\end{enumerate}
\end{theorem}
\begin{proof}
All statements follow from \eqref{N1 equivalent 1}, \eqref{core subtractivity} and the fact $(-A)^{\core}=-A^{\core}$. 
\end{proof}

\section*{Statements and Declarations}

\noindent {\bf Competing Interests}
 
\noindent Not applicable.

\noindent {\bf  Funding} 

\noindent This work was supported by Universidad Nacional de R\'{\i}o Cuarto (Grant PPI 18/C559),
CONICET (Grant PIP 112-202001-00694CO),  CONICET (Grant PIBAA 28720210100658CO), and by Universidad Nacional de La Pampa, Facultad de Ingenier\'ia (Grant Resol. Nro. 135/19).

\end{document}